\documentclass[12pt]{amsart}
\usepackage{amssymb}

\theoremstyle{plain}
 \newtheorem{thm}{Theorem}
 \newtheorem{cor}{Corollary}
 \newtheorem{lem}{Lemma}
 
\theoremstyle{definition}
 \newtheorem{defn}{Definition}

\theoremstyle{remark}
 \newtheorem{rem}{Remark}

\allowdisplaybreaks

\begin{document}
\title[Common fixed points for three or four
 mappings]
{Common fixed points for three or four
 mappings\\ via common fixed point
for two mappings}
\author[C. Di Bari]{Cristina Di Bari}
\address[C. Di Bari]{Universit\`{a} degli Studi di Palermo \\
 Dipartimento di Matematica e Informatica \\
Via Archirafi, 34 - 90123 Palermo (Italy)}
\email{dibari@math.unipa.it}
\author[P. Vetro]{Pasquale Vetro}
\address[P. Vetro]{Universit\`{a} degli Studi di Palermo \\
 Dipartimento di Matematica e Informatica \\
Via Archirafi, 34 - 90123 Palermo (Italy)}
 \email{vetro@math.unipa.it}
\thanks{The authors are supported  by Universit\`{a} degli Studi di Palermo, R. S. ex 60\%.}
\date{}


\begin{abstract}In this paper, we shall show that some coincidence point and common
fixed point results for three or four mappings could easily be
obtained from the corresponding fixed point results for two
mappings.

\smallskip



\end{abstract}
\maketitle

\section{Introduction and Preliminaries}
\noindent Fixed point theory is an important and actual topic of
nonlinear analysis [1-7,\,9-12]. Moreover, it's well known that the
contraction mapping principle, formulated and proved in the Ph.D.
dissertation of Banach in 1920 which was published in 1922 is one of
the most important theorems in classical functional analysis.

It is also known that common fixed point theorems are
generalizations of fixed point theorems. Over the past few decades,
there have many researchers interested in generalizing fixed point
theorems to coincidence point theorems and common fixed point
theorems. Recently, Haghi et al. \cite{Ha} have shown that some
coincidence point and common fixed point results for two mappings
are not real generalizations as they could easily be obtained from
the corresponding fixed point theorems. They used the following
Lemma that is a consequence  of the axiom of choise.
\begin{lem}[\cite{Ha}, Lemma 2.1] \label{L1}Let $X$ be a nonempty set and $f : X \to X$ a function.
Then there exists a subset $E \subset X$ such that $f E = f X$ and
$f : E \to X$ is one-to-one.
\end{lem}

In this paper, we shall show that some coincidence point and common
fixed point for three or four mappings could easily be obtained from
the corresponding fixed point results for two mappings.

\section{Main results}
In this section, first we prove a result of common fixed point for
two self-mappings satisfying a contractive condition of Berinde type
\cite{Be}. Successively, we deduce some common fixed point results
for three or four self-mappings.
\begin{thm}\label{T2} Let $(X,d)$ be a metric space and $S$ and $T$
self-mappings on $X$. Assume that the following condition holds:
\begin{align}d(Sx,Ty) \leq &\alpha d(x,Sx)+ \beta d(y,Ty)+ \gamma
d(x,y)+ \delta [d(y,Sx)+d(x,Ty)]\\
&+ L \min\{d(x,Sx),d(y,Ty), d(y,Sx), d(x,Ty) \}\nonumber
\end{align}
for all $x,y \in X$, where $\alpha, \beta, \gamma, \delta \in [0,1[$
are such that $\alpha +  \beta + \gamma +2 \delta <1$ and $L\geq0$.
If $SX$ or $TX$ is a complete subspace of $X$, then $S$ and $T$ have
a unique common fixed point in $X$.
\end{thm}
\begin{proof} Let $x_0 \in X$. We construct the sequence
$\{x_n\}$ such that $$x_1=Sx_0,\, x_2=Tx_1, \ldots,
x_{2n+1}=Sx_{2n},\, x_{2n+2}=Tx_{2n+1},\ldots \, .$$

We show that $\{x_n\}$ is a Cauchy sequence in $X$. Taking
$x=x_{2n}$ and $y=x_{2n+1}$ in (1), we obtain
\begin{align*}d(x_{2n+1},x_{2n+2}) = &d(Sx_{2n},Tx_{2n+1})
\leq \alpha d(x_{2n+1},x_{2n})+\beta d(x_{2n+1},x_{2n+2})\\ &
+\gamma d(x_{2n},x_{2n+1})
+\delta[d(x_{2n+1},x_{2n+1})+d(x_{2n},x_{2n+2})]\\
& +L \min\{d(x_{2n},x_{2n+1}),d(x_{2n+1},x_{2n+2}),
d(x_{2n+1},x_{2n+1}),d(x_{2n},x_{2n+2})\}.
\end{align*}

Consequently, we have $$(1-\beta-\delta)d(x_{2n+1},x_{2n+2}) \leq
(\alpha+\gamma+\delta)d(x_{2n},x_{2n+1}).$$

Similarly, Taking $x=x_{2n+2}$ and $y=x_{2n+1}$ in (1), we get
$$(1-\alpha-\delta)d(x_{2n+3},x_{2n+2}) \leq
(\beta+\gamma+\delta)d(x_{2n+1},x_{2n+2}).$$

Let $k=
\max\{\dfrac{\alpha+\gamma+\delta}{1-\beta-\delta},\dfrac{\beta+\gamma+\delta}{1-\alpha-\delta}\}<1$.
Then, we have $$d(x_{n+1},x_{n})\leq kd(x_{n},x_{n-1})\leq \cdots
\leq k^nd(x_1,x_0).$$

It implies that $\{x_n\}$ is a Cauchy sequence. If $SX$ is complete,
then $x_{2n-1} \to z \in SX$ and hence $x_n \to z$. (The same holds
if $TX$ is complete with $z \in TX$.) Now, we prove that $z=Sz$. If
not, we have
\begin{align*} d(Sz,z) & \leq d(Sz,Tx_{2n+1})+d(Tx_{2n+1},z)\\
& \leq d(x_{2n+2},z) + \alpha d(Sz,z) +\beta d(x_{2n+1},x_{2n+2}) +
\gamma d(x_{2n+1},z)\\ & \ \ \ \ +\delta
[d(Sz,x_{2n+1})+d(x_{2n+2},z)]\\
& \ \ \ \
+L\min\{d(Sz,z),d(x_{2n+1},x_{2n+2}),d(Sz,x_{2n+1}),d(x_{2n+2},z)\},
\end{align*}
as $n \to +\infty$, we get $$d(Sz,z)  \leq (\alpha
+\delta)d(Sz,z).$$  This is a contradiction and hence $z=Sz$.

Similarly, we deduce that $z=Tz$ and so $z$ is a common fixed point
of $S$ and $T$. The uniqueness follows by (1).
\end{proof}

\begin{rem}Clearly, Theorem \ref{T2} holds if we assume that $X$ is a
complete metric space.
\end{rem}

Let $X$ be a non-empty set and $T, f: X \to X$. The mappings $T, f$
are said to be weakly compatible if they commute at their
coincidence point (i.e. $Tfx = fTx$ whenever $Tx = fx$). A point $y
\in X$ is called point of coincidence of $T$ and $f$ if there exists a point $%
x\in X\ $such that $y=Tx=fx$.

\begin{lem}\label{LemVe}
Let\ $X$ be a non-empty set and the mappings $T,~f:X\rightarrow X\
$have a unique point$\ $of coincidence $v\ $in$\ X.\ $If $T$ and $f$
are weakly compatible, then $T$ and $f$ have a unique common fixed
point.
\end{lem}

\begin{proof}
Since\ $v$ is a point of\ coincidence of $T$ and $f$. Therefore, $%
v=fu=Tu$\ for\ some $u\in X.$ By weakly compatibility of  $T$ and $f$ we have%
\begin{equation*}
 Tv=Tfu=fTu=fv.
\end{equation*}%
It implies that $Tv=fv=w$ (say). Then $w$ is a point of coincidence
of $T$ and $f$. Therefore, $v=w$ by uniqueness. Thus $v$ is\ a
unique common fixed point of\ $T$ and $f$.
\end{proof}

From Theorem \ref{T2}, we deduce the following theorems.

\begin{thm}\label{T3} Let $(X,d)$ be a metric space and $S, T$ and $f$
self-mappings on $X$ such that $SX\cup TX \subset fX$. Assume that
the following condition holds:
\begin{align}d(Sx,Ty) \leq & \alpha d(fx,Sx)+ \beta d(fy,Ty)+ \gamma
d(fx,fy)\\ &+ \delta [d(fy,Sx)+d(fx,Ty)] \nonumber\\
&+L\min\{d(fx,Sx),d(fy,Ty),d(fy,Sx),d(fx,Ty)\},\nonumber
\end{align}
for all $x,y \in X$, where $\alpha, \beta, \gamma, \delta \in [0,1[$
are such that $\alpha +  \beta + \gamma +2 \delta <1$ and $L\geq 0$.
If $(S,f)$ and $(T,f)$ are weakly compatible and $fX$ is a complete
subspace of $X$, then $S,T$ and $f$ have a unique common fixed point
in $X$.
\end{thm}
\begin{proof} By Lemma \ref{L1} there exists $E\subset X$ such that $fE=fX$
and $f : E \to X$ is one-to-one. Now, define two mappings $g,h: fE
\to fE$ by $g(fx)=Sx$ and $h(fx)=Tx$, respectively. Since $f$ is
one-to-one on $E$, then $g,h$ are well-defined. Note that,
\begin{align}d(g(fx),h(fy)) & \leq \alpha d(fx,g(fx))+ \beta d(fy,h(fy))+ \gamma
d(fx,fy)\\ & \ \ \ + \delta [d(fy,g(fx))+d(fx,h(fy))]\nonumber\\
& \ \ \
+L\min\{d(fx,g(fx)),d(fy,h(fy)),d(fy,g(fx)),d(fx,h(fy))\}.\nonumber
\end{align}
By Theorem \ref{T2}, as $fE$ is a complete subspace of $X$, we
deduce that there exists a unique common fixed point $fz \in fE$ of
$g$ and $h$, that is $fz=g(fz)=h(fz)$. Thus $z$ is a coincidence
point of $S,T$ and $f$. Note that, if $Sx=Tx=fx$, using (2) we get
$fx=fz$ and so $S,T$ and $f$ have a unique point of coincidence $fz
\in fE$. Now, since $(S,f)$ and $(T,f)$ are weakly compatible, by
Lemma \ref{LemVe}, we deduce that $fz$ is a unique fixed point of
$S,T$ and $f$.
\end{proof}

\begin{thm}\label{T4} Let $(X,d)$ be a metric space and $S,T,f$ and $g$
self-mappings on $X$ such that $SX,\,TX \subset fX=gX$. Assume that
the following condition holds:
\begin{align}d(Sx,Ty) \leq &\alpha d(fx,Sx)+ \beta d(gy,Ty)+ \gamma
d(fx,gy)+ \delta [d(gy,Sx)+d(fx,Ty)]\\
&+L\min\{d(fx,Sx),d(gy,Ty),d(gy,Sx),d(fx,Ty)\},\nonumber
\end{align}
for all $x,y \in X$, where $\alpha, \beta, \gamma, \delta \in [0,1[$
are such that $\alpha +  \beta + \gamma +2 \delta <1$ and $L\geq 0$.
If $(S,f)$ and $(T,g)$ are weakly compatible and $fX$ is a complete
subspace of $X$, then $S,T,f$ and $g$ have a unique common fixed
point in $X$.
\end{thm}
\begin{proof} By Lemma \ref{L1}, there exists $E_1, E_2\subset X$ such that
$fE_1=fX=gX=gE_2$,\break  $f : E _1\to X$  and $g : E _2\to X$ are
one-to-one. Now, define two mappings $A,B: fE_1 \to fE_1$ by
$A(fx)=Sx$ and $B(gx)=Tx$, respectively. Since $f,g$ are one-to-one
on $E_1, E_2$, respectively, then the mappings $A,B$ are
well-defined. Note that,
\begin{align*}d(Afx,Bgy) \leq &\alpha d(fx,Afx)+ \beta d(gy,Bgy)+ \gamma
d(fx,gy)+ \delta [d(gy,Afx)+d(fx,Bgy)]\\
&+L\min\{d(fx,Afx),d(gy,Bgy),d(gy,Afx),d(fx,Bgy)\},
\end{align*}
for all $fx,gy \in fE_1$.

By Theorem \ref{T2}, as $fE_1$ is a complete subspace of $X$, we
deduce that there exists a unique common fixed point $fz \in fE_1$
of $A$ and $B$. Thus $z$ is a coincidence point of $S$ and $f$, that
is $Sz=fz$. Now, let $v \in X$ such that $fz=gv$, then $v$ is a
coincidence point of $T$ and $g$, that is $Tv=gv$. We show that $S$
and $f$ have a unique point of coincidence. If $Sw=fw$ and $fw\neq
fz$, from (4) for $x=w$ and $y=v$, we get $$d(fw,gv)=d(Sw,Tv) \leq
(\gamma +2\delta) d(fw,gv),$$ which implies $fw=gv=fz$. Since $S$
and $f$ are weakly compatible, by Lemma \ref{LemVe}, it follows that
$fz$ is the unique common fixed point of $S$ and $f$. Similarly,
$gv$ is the unique common fixed point for $T$ and $g$ and so $fz=gv$
is the unique common fixed point for $S,T,f$ and $g$.
\end{proof}

From the proof of the Theorems 2-3 it follows that the following
results of points of coincidence for three or four self-mappings
hold.

\begin{thm}\label{C3} Let $(X,d)$ be a metric space and $S, T$ and $f$
self-mappings on $X$ such that $SX\cup TX \subset fX$. Assume that
the following condition holds:
\begin{align}d(Sx,Ty) \leq & \alpha d(fx,Sx)+ \beta d(fy,Ty)+ \gamma
d(fx,fy)\\ &+ \delta [d(fy,Sx)+d(fx,Ty)] \nonumber\\
&+L\min\{d(fx,Sx),d(fy,Ty),d(fy,Sx),d(fx,Ty)\},\nonumber
\end{align}
for all $x,y \in X$, where $\alpha, \beta, \gamma, \delta \in [0,1[$
are such that $\alpha +  \beta + \gamma +2 \delta <1$ and $L\geq 0$.
If $fX$ is a complete subspace of $X$, then $S,T$ and $f$ have a
unique point of coincidence in $X$.
\end{thm}

\begin{thm}\label{C4} Let $(X,d)$ be a metric space and $S,T,f$ and $g$
self-mappings on $X$ such that $SX,\,TX \subset fX=gX$. Assume that
the following condition holds:
\begin{align}d(Sx,Ty) \leq &\alpha d(fx,Sx)+ \beta d(gy,Ty)+ \gamma
d(fx,gy)+ \delta [d(gy,Sx)+d(fx,Ty)]\\
&+L\min\{d(fx,Sx),d(gy,Ty),d(gy,Sx),d(fx,Ty)\},\nonumber
\end{align}
for all $x,y \in X$, where $\alpha, \beta, \gamma, \delta \in [0,1[$
are such that $\alpha +  \beta + \gamma +2 \delta <1$ and $L\geq 0$.
If $fX$ is a complete subspace of $X$, then $S,T,f$ and $g$ have a
unique point of coincidence in $X$.
\end{thm}

\end{document}